\documentclass{article}%
\usepackage{amsmath}
\usepackage{amsfonts}
\usepackage{amssymb}
\usepackage{graphicx}
\usepackage[ruled,vlined,linesnumbered]{algorithm2e}%
\providecommand{\U}[1]{\protect\rule{.1in}{.1in}}
\newtheorem{theorem}{Theorem} [section]

\newtheorem{conjecture}[theorem]{Conjecture}
\newtheorem{corollary}[theorem]{Corollary}

\newtheorem{example}[theorem]{Example}

\newtheorem{problem}[theorem]{Problem}

\newenvironment{proof}[1][Proof]{\noindent\textbf{#1.} }{\ \rule{0.5em}{0.5em}}
\begin{document}
  
\author{Vadim E. Levit \\Department of Mathematics\\Ariel University, Ariel, Israel\\levitv@ariel.ac.il \\ \\David Tankus\\Department of Software Engineering\\Sami Shamoon College of Engineering, Ashdod, Israel\\davidt@sce.ac.il}
\title{Recognizing $\mathbf{W_2}$ Graphs}
\date{}
\maketitle

\begin{abstract}
Let $G$ be a graph. A set $S \subseteq V(G)$ is {\it independent} if its elements are pairwise nonadjacent. A vertex $v \in V(G)$ is {\it shedding} if for every independent set $S \subseteq V(G) \setminus N[v]$ there exists $u \in N(v)$ such that $S \cup \{u\}$ is independent. An independent set $S$ is {\it maximal} if it is not contained in another independent set. An independent set $S$ is {\it maximum} if the size of every independent set of $G$ is not bigger than $|S|$. The size of a maximum independent set of $G$ is denoted $\alpha(G)$. A graph $G$ is {\it well-covered} if all its maximal independent sets are maximum, i.e. the size of every maximal independent set is $\alpha(G)$. The graph $G$ belongs to class $\mathbf{W_2}$ if every two pairwise disjoint independent sets in $G$ are included in two pairwise disjoint maximum independent sets. If a graph belongs to the class $\mathbf{W_2}$ then it is well-covered.

Finding a maximum independent set in an input graph is an NP-complete problem. Recognizing well-covered graphs is co-NP-complete. The complexity status of deciding whether an input graph belongs to the $\mathbf{W_2}$ class is not known. Even when the input is restricted to well-covered graphs, the complexity status of recognizing graphs in $\mathbf{W_2}$ is not known.

In this article, we investigate the connection between shedding vertices and $\mathbf{W_2}$ graphs. On the one hand, we prove that recognizing shedding vertices is co-NP-complete. On the other hand, we find polynomial solutions for restricted cases of the problem. We also supply polynomial characterizations of several families of $\mathbf{W_2}$ graphs.
\end{abstract}

\section{Introduction}
\subsection{The classes $\mathbf{W_k}$}
An {\it independent set} of vertices in a graph $G$ is a set of vertices $S \subseteq V(G)$ whose elements are pairwise nonadjacent. An independent set is {\it maximal} if is is not a subset of another independent set. An independent set is {\it maximum} if $G$ does not contain an independent set of a higher cardinality. The cardinality of a maximum independent set in $G$ is denoted $\alpha(G)$. A graph $G$ is {\it well-covered} if all its maximal independent sets are maximum, i.e. the size of every maximal independent set is $\alpha(G)$.

\begin{problem}
\label{wcrecog}
{\bf WC}\\
Input: A graph $G$.\\
Question: Is $G$ well-covered?
\end{problem}

Finding a maximum independent set in an input graph is known to be an NP-complete problem. However, finding a maximal independent set in an input graph can be done polynomially using the greedy algorithm. If the input is restricted to well-covered graphs then the greedy algorithm for finding a maximal independent set always yields a maximum independent set. Hence, finding a maximum independent set in a well-covered graph is a polynomial task. Recognizing well-covered graphs is known to be in co-NP-complete \cite{cs:note} \cite{sknryn:compwc}.

Let $k$ be a positive integer. A graph $G$ belongs to class $\mathbf{W_k}$ if every $k$ pairwise
disjoint independent sets in $G$ are included in $k$ pairwise disjoint maximum independent
sets \cite{Staples:thesis}. It holds that $\mathbf{W_1} \supseteq \mathbf{W_2} \supseteq \mathbf{W_3} \supseteq \ldots$, where $\mathbf{W_1}$ is the family of all well-covered graphs. The {\bf W2} problem and the {\bf WCW2} problem are defined as follows. Their complexity statuses are still open.

\begin{problem}
\label{w2recog}
{\bf W2}\\
Input: A graph $G$.\\
Question: Is $G \in \mathbf{W_2}$ ?
\end{problem}

\begin{problem}
\label{wcw2recog}
{\bf WCW2}\\
Input: A well-covered graph $G$.\\
Question: Is $G \in \mathbf{W_2}$ ?
\end{problem}

In \cite{dlm:disjoint} graphs with two disjoint maximum independent sets are studied.
\begin{theorem}
\label{disjointmis}
\cite{dlm:disjoint}
Let $G$ be a graph. The following assertions are equivalent:
\begin{enumerate}
\item $G$ admits two disjoint maximum independent sets.
\item There exists a matching $M$ of size $\alpha(G)$ such that $G[V (M)]$ is a bipartite graph. 
\item There exists a set $A \subseteq V(G)$ such that $G \setminus A$ is a bipartite graph having a perfect matching of size $\alpha(G)$.
\end{enumerate}
\end{theorem}

A vertex $v \in V(G)$ is {\it shedding} if for every independent set $S \subseteq V(G) \setminus N[v]$ there exists $u \in N(v)$ such that $S \cup \{u\}$ is independent. Equivalently, $v$ is shedding if there does not exist an independent set in $V(G) \setminus N[v]$ which dominates $N(v)$ \cite{w:shed}. The {\bf SHED} problem is defined as follows.

\begin{problem}
\label{shedprob}
{\bf SHED}\\
Input: A graph $G$ and a vertex $v \in V(G)$.\\
Question: Is $v$ shedding ?
\end{problem}

Let $Shed(G)$ denote the set of all shedding vertices of $G$.

\begin{theorem}
\label{w2shed}
\cite{lm:shed,lm:w2}
For every graph $G$ having no isolated vertices, the following assertions are equivalent:
\begin{enumerate}
\item $G$ is in the class $\mathbf{W_2}$.
\item $G \not\approx P_{3}$ and $G \setminus v$ is well-covered, for every $v \in V(G)$.
\end{enumerate}
\end{theorem}

\begin{theorem}
\label{w2wcshed}
\cite{lm:shed,lm:w2}
For every well-covered graph $G$ having no isolated vertices, the following assertions are equivalent:
\begin{enumerate}
\item $G$ is in the class $\mathbf{W_2}$.
\item $G \setminus N[v]$ is in the class $\mathbf{W_2}$, for every $v \in V(G)$.
\item $Shed(G) = V(G)$.
\end{enumerate}
\end{theorem}

This article finds polynomial solutions for several restricted cases of Problem \ref{w2recog} and Problem \ref{wcw2recog}. We also prove that Problem \ref{shedprob} is co-NP-complete.

\subsection{Definitions and notation}
Throughout this article the following notation and definitions are used. Graphs are undirected, simple, and loopless. The number of vertices in a graph $G$ is $n = |V(G)|$, and the number of edges is $m = |E(G)|$. If $G$ contains cycles than the {\it girth} of $G$ is the minimal length of a cycle in the graph. Otherwise, the {\it girth} of $G$ is infinite.

Let $v \in V(G)$. The set of vertices adjacent to $v$ in $G$ is denoted $N_{G}(v)$. Define $N_{G}[v] = \{v\} \cup N_{G}(v)$. The {\it degree} of $v$ in $G$ is $d_{G}(v) = |N_{G}(v)|$. If $G$ is the only graph mentioned in the context then $N_{G}(v)$, $N_{G}[v]$ and $d_{G}(v)$ are abbreviated to $N(v)$, $N[v]$ and $d(v)$, respectively. For every $S \subseteq V(G)$ denote $N[S] = \cup_{v \in S} N[v]$ and $N(S) = N[S] \setminus S$.

Let $S$ and $T$ be sets of vertices of $G$. Then $S$ {\it dominates} $T$ if $T \subseteq N[S]$. If $S$ dominates $V(G)$ then $S$ is a dominating set. A set $S \subseteq V(G)$ is a {\it clique} if every two vertices of $S$ are adjacent to each other. A vertex $v \in V(G)$ is {\it simplicial} if $N[v]$ is a clique.

\section{Recognizing shedding vertices}
\subsection{Co-NP-completeness}
Let $ {\cal X}=\{x_{1},\ldots x_{n}\}$ be a set of 0-1 variables. We define the set of
\textit{literals} $L_{\cal X}$ over $\cal X$ by $L_{\cal X} = \{x_{i}, \overline{x_{i}} : i = 1,\ldots ,n\}$, where $\overline{x} = 1 - x$ is the \textit{negation} of $x$. A
\textit{truth assignment} to $\cal X$ is a mapping $t:{\cal X} \longrightarrow\{0,1\}$
that assigns a value $t(x_{i}) \in\{0,1\}$ to each variable $x_{i} \in \cal X$. We
extend $t$ to $L_{\cal X}$ by putting $t(\overline{x_{i}}) = \overline{t(x_{i})}$.
A literal $l \in L_{\cal X}$ is true under $t$ if $t(l) = 1$. A \textit{clause}
over $\cal X$ is a conjunction of some literals of $L_{\cal X}$, such that for every variable $x \in \cal X$, the clause contains at most one literal out of $x$ and its negation. Let ${\cal C}=\{c_{1},\ldots ,c_{m}\}$ be a set of clauses over $\cal X$. A truth assignment $t$ to $\cal X$
\textit{satisfies} a clause $c_{j} \in \cal C$ if $c_{j}$ contains at least one
true literal under $t$. The {\bf SAT} problem, defined as follows, is a well-known \textbf{NP}-complete problem \cite{gj:NPC}.

\begin{problem}
\label{satprob}
{\bf SAT}\\
Input: A set of variables ${\cal X}=\{x_{1}, \ldots ,x_{n}\}$, and a set of clauses ${\cal C}=\{c_{1}, \ldots ,c_{m}\}$ over $\cal X$.\\
Question:  Is there a truth assignment to $\cal X$ which satisfies all
clauses of $\cal C$?
\end{problem}

\begin{theorem}
\cite{gj:NPC}
\label{satNPC} The {\bf SAT} problem is NP-complete.
\end{theorem}

The {\bf WCSHED} problem, defined as follows, is closely related to the {\bf SHED} problem.

\begin{problem}
\label{wcshedrecog}
{\bf WCSHED}\\
Input: A well-covered graph $G$ and a vertex $v \in V(G)$.\\
Question: Is $v$ shedding ?
\end{problem}

In the {\bf WCSHED} problem, the input includes a graph $G$ such that for every $x \in V(G)$ and for every maximal independent set $S$ of $V(G) \setminus N[x]$ it holds that $\alpha(N(x) \setminus N(S)) \leq 1$. The question is whether for a specific vertex $v \in V(G)$ it holds that $\alpha(N(v) \setminus N(S)) = 1$ for every maximal independent set $S$ of $V(G) \setminus N[v]$. 

The complexity status of the {\bf WCSHED} problem is not known. However, the complexity status of the {\bf SHED} problem is found by Theorem \ref{shedc3conpc}.

\begin{theorem}
\label{shedc3conpc}
The {\bf SHED} problem is co-NP-complete, even if its input is restricted to graphs without cycles of length 3.
\end{theorem}

\begin{proof}
Let $I=(G, v)$ be an instance of the {\bf SHED} problem. An independent set $S \subseteq N_{2}(v)$ which dominates $N(v)$, if exists, is a witness that $v$ is not a shedding vertex, and $I$ is negative. Therefore, the problem is co-NP. We prove co-NP-completeness by a supplying a polynomial reduction from the {\bf SAT} problem to the complement of the {\bf SHED} problem.

Let $I_{1} = ({\cal X}=\{x_{1},\ldots,x_{n}\}, {\cal C}=\{c_{1},\ldots,c_{m}\})$  be an instance of the {\bf SAT} problem. Define a graph $G$ as follows.  
\[V(G) = \{v\} \cup \{v_{i} : 1 \leq i \leq m\} \cup U,\] where $U = \{ u_{i,l} : l \in c_{i} \}$.
\[E(G) = \{vv_{i} : 1 \leq i \leq m\} \cup \{v_{i}u_{i,l} : l \in c_{i} \} \cup E(U),\] 
where $E(U) = \{ u_{i,l}u_{j,\overline{l}} : l \in c_{i}, \overline{l} \in c_{j} \}$.

The subgraph induced by $U$ is bipartite, since every edge connects instances of a literal and its negation. Moreover,  $v$ is not on a triangle since its neighbors are independent. A vertex $v_{i}$ is not on a triangle since the clause $c_{i}$ does not contain both a literal and its negation. Therefore, $G$ does not contain cycles of length 3.

Let $I_{2} = (G, v)$ be an instance of the complement of the {\bf SHED} problem. It remains to prove that $I_{1}$ and $I_{2}$ are eqiuvalent.

Suppose $I_{1}$ is positive. There exists a truth assignment, $t$, to $\cal X$ which satisfies all clauses of $\cal C$. Let $S = \{ u_{i,l} : l \in c_{i}, t(l) = 1 \}$. The fact that $S$ does not contain both a literal and its negation implies that $S$ is independent. The fact that $t$ satisfies all clauses of $\cal C$ implies that $S$ dominates $N(v) = \{v_{i} : 1 \leq i \leq m\}$. Therefore, $v$ is not shedding, and $I_{2}$ is positive.

On the other hand, if $I_{2}$ is positive there exists an independent set $S \subseteq N_{2}(v)$ which dominates $N(v)$. Let $t$ be a truth assignment such that $t(l)=1$ for every literal $l$ which has an instance $u_{l} \in S$. This assignment is valid since $S$ does not contain instances of both a literal and its negation. The fact that $S$ dominates $N(v) = \{v_{i} : 1 \leq i \leq m\}$ implies that all clauses are satisfied by $t$. Therefore, $I_{1}$ is positive. 
\end{proof}

\begin{example}
$I_{1} = ({\cal X},{\cal C} )$ is an instance of the {\bf SAT} problem, where ${\cal X} = \{x_{1}, x_{2}, x_{3}, x_{4}\}$, 
\[{\cal C} = \{c_{1} = (x_{1}, \overline{x_{2}}, x_{3}), c_{2}=(x_{2},x_{4}), c_{3}=(\overline{x_{1}}, \overline{x_{3}}, \overline{x_{4}}), c_{4}=(x_{1}, \overline{x_{2}}, \overline{x_{3}}, \overline{x_{4}})\}\] 
$I_{2} = (G, v)$ is an equivalent instance of the complement of the {\bf SHED} problem, where $G$ is the graph shown in Figure \ref{satfig}.

The satisfying truth assignment $t(x_{1}) = t(x_{2}) = 1$, $t(x_{3}) = t(x_{4})=0$ corresponds to the independent set $\{u_{1,x_{1}}, u_{2,x_{2}}, u_{3,\overline{x_{3}}},u_{3,\overline{x_{4}}}, u_{4,x_{1}}, u_{4, \overline{x_{3}}}, u_{4,\overline{x_{4}}}\} \subseteq N_{2}(v)$ which dominates $N(v)$.
\end{example}

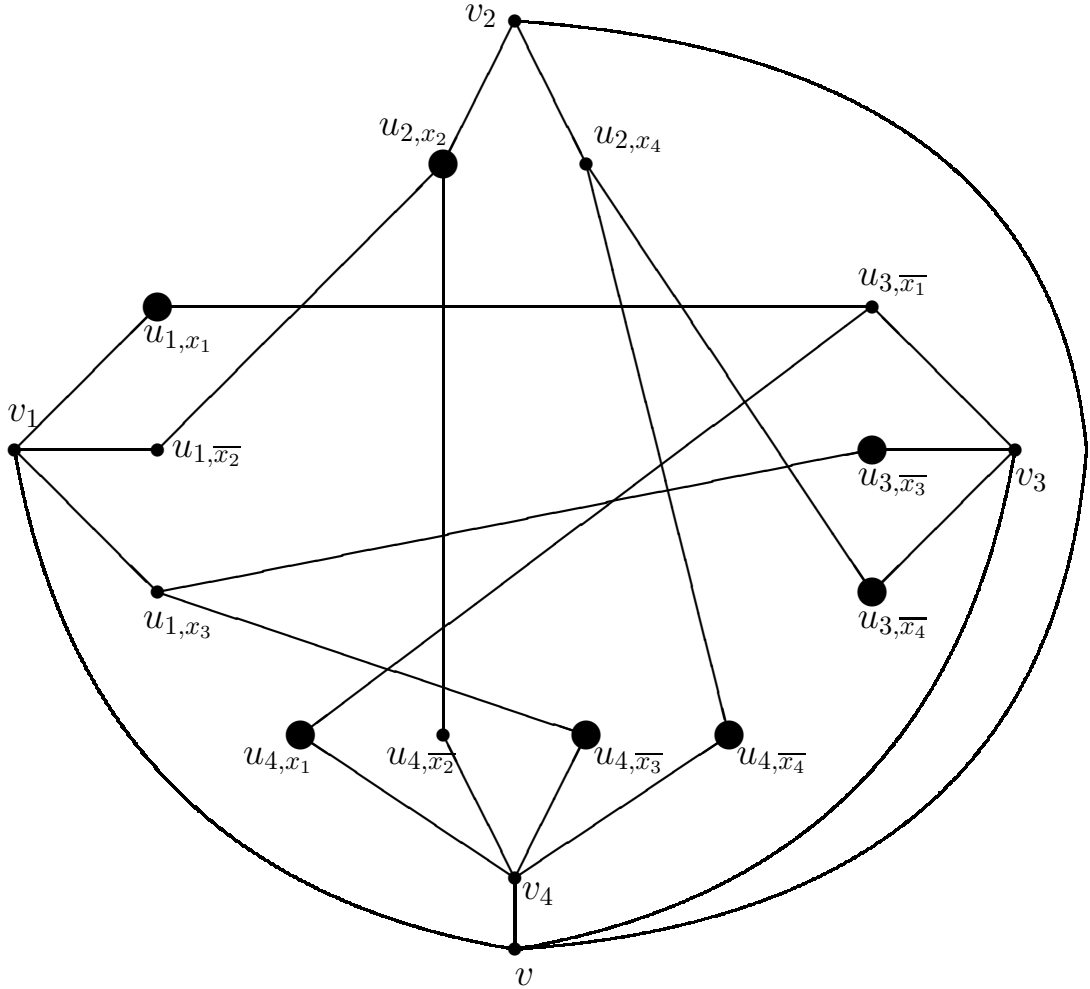
\begin{figure}[h]
\setlength{\unitlength}{0.95cm} \begin{picture}(14,14)\thicklines
\multiput(6,12)(2,0){2}{\circle*{0.2}}
\multiput(4,4)(2,0){4}{\circle*{0.2}}
\multiput(2,10)(0,-2){3}{\circle*{0.2}}
\multiput(12,10)(0,-2){3}{\circle*{0.2}}
\put(7,14){\circle*{0.2}}
\put(7,2){\circle*{0.2}}
\put(0,8){\circle*{0.2}}
\put(14,8){\circle*{0.2}}
\put(-0.1,8.4){\mbox{\Large $v_{1}$}}
\put(14,7.5){\mbox{\Large $v_{3}$}}
\put(6.3,14){\mbox{\Large $v_{2}$}}
\put(7.1,1.7){\mbox{\Large $v_{4}$}}
\put(0,8){\line(1,0){2}}
\put(0,8){\line(1,-1){2}}
\put(0,8){\line(1,1){2}}
\put(14,8){\line(-1,0){2}}
\put(14,8){\line(-1,-1){2}}
\put(14,8){\line(-1,1){2}}
\put(7,2){\line(1,2){1}}
\put(7,2){\line(3,2){3}}
\put(7,2){\line(-1,2){1}}
\put(7,2){\line(-3,2){3}}
\put(7,14){\line(-1,-2){1}}
\put(7,14){\line(1,-2){1}}
\put(1.8,9.5){\mbox{\Large $u_{1,x_{1}}$}}
\put(2.2,7.9){\mbox{\Large $u_{1,\overline{x_{2}}}$}}
\put(1.8,5.5){\mbox{\Large $u_{1,x_{3}}$}}
\put(11.8,10.3){\mbox{\Large $u_{3,\overline{x_{1}}}$}}
\put(11.8,7.5){\mbox{\Large $u_{3,\overline{x_{3}}}$}}
\put(11.8,5.5){\mbox{\Large $u_{3,\overline{x_{4}}}$}}
\put(5.1,12.4){\mbox{\Large $u_{2,x_{2}}$}}
\put(8.1,12.3){\mbox{\Large $u_{2,x_{4}}$}}
\put(3.2,3.6){\mbox{\Large $u_{4,x_{1}}$}}
\put(5.2,3.6){\mbox{\Large $u_{4,\overline{x_{2}}}$}}
\put(8.1,3.6){\mbox{\Large $u_{4,\overline{x_{3}}}$}}
\put(10.1,3.6){\mbox{\Large $u_{4,\overline{x_{4}}}$}}
\put(2,10){\line(1,0){10}}
\put(12,10){\line(-4,-3){8}}
\put(2,8){\line(1,1){4}}
\put(2,6){\line(5,1){10}}
\put(2,6){\line(3,-1){6}}
\put(6,12){\line(0,-1){8}}
\put(8,12){\line(2,-3){4}}
\put(8,12){\line(1,-4){2}}
\put(7,1){\circle*{0.2}}
\put(7,1){\line(0,1){1}}
\qbezier(7,1)(13,2)(14,8)
\qbezier(7,1)(1,2)(0,8)
\qbezier(7,1)(14.5,1.5)(15,8)
\qbezier(7,14)(14.5,13.5)(15,8)
\put(7,0.5){\mbox{\Large $v$}}
\put(2,10){\circle*{0.4}}
\put(6,12){\circle*{0.4}}
\put(12,8){\circle*{0.4}}
\put(12,6){\circle*{0.4}}
\put(4,4){\circle*{0.4}}
\put(8,4){\circle*{0.4}}
\put(10,4){\circle*{0.4}}
\end{picture}\caption{The graph $G$.}%
\label{satfig}
\end{figure}

\subsection{Non-shedding verices and relating edges}
Let $G$ be a graph and $xy \in E(G)$. Then $xy$ is {\it relating} if there exists an independent set $S \subseteq V(G) \setminus N[\{x,y\}]$ such that each of $S \cup \{x\}$ and $S \cup \{y\}$ is a maximal independent set of $G$. The {\bf RE} problem, defined as follows, is known to be NP-complete.

\begin{problem}
\label{reprob}
{\bf RE}\\
Input: A graph $G$ and an edge $e \in E(G)$.\\
Question: Is $e$ relating ?
\end{problem}

\begin{theorem}
\label{renpc}
\cite{bnz:wcc4}
The {\bf RE} problem is NP-complete.
\end{theorem}

Non-shedding verices and relating edges are closely related notions. A witness for their existence is an independent set of vertices, which dominates all vertices of the graph except the non-shedding vertex or the endpoints of the relating edge. Moreover, one can prove that {\bf SHED} problem is co-NP-complete by a polynomial reduction from the {\bf RE} problem to the complement of the {\bf SHED} problem.
We do not present the reduction and the proof of its correctness because this result is a restricted case of Theorem \ref{shedc3conpc}, in which cycles of length 3 are forbidden. Also Theorem \ref{shedrec6} shows the connection between non-shedding vertices and relating edges.

\begin{theorem}
\label{shedrec6}
Let $G$ be a graph without cycles of lengths 4, 5 and 6, and $xy \in E(G)$. Suppose $N(x) \cap N(y) = \varnothing$, $d(x) \geq 2$ and $d(y) \geq 2$. The following assertions are equivalent:
\begin{enumerate}
\item None of $x$ and $y$ is a shedding vertex.
\item $xy$ is a relating edge.
\end{enumerate}
\end{theorem}

\begin{proof}

$1 \implies 2$  Since $x$ is not shedding, there exists an independent set $S_{x} \subseteq N_{2}(x)$ which dominates $N(x)$. Similarly, there exists an independent set $S_{y} \subseteq N_{2}(y)$ which dominates $N(y)$. The fact that $G$ does not contain cycles of lengths 4 implies that $S_{x} \cap N(y) = \varnothing$ and $S_{y} \cap N(x) = \varnothing$. Define $S = S_{x} \cup S_{y}$. Then $S \cap N(\{x,y\}) = \varnothing$. Assume, on the contrary, that $S$ is not independent. Then there exists $x'' \in S_{x} \setminus N(y)$ and $y'' \in S_{y} \setminus N(x)$ such that $x''y'' \in E(G)$. Let $x' \in N(x'') \cap N(x)$ and $y' \in N(y'') \cap N(y)$. Then $(x'',x',x,y,y',y'')$ is a cycle of length 6, which is a contradiction. Therefore, $S$ is independent. Let $S^{*}$ be a maximal independent set of $G \setminus N[\{x,y\}]$ which contains $S$. Then each of $S^{*} \cup \{x\}$ and $S^{*} \cup \{y\}$ is a maximal independent set of $G$. Therefore,  $xy$ is a relating edge.

$2 \implies 1$  There exists an independent set $S$ such that each of $S \cup \{x\}$ and $S \cup \{y\}$ is a maximal independet set of $G$. Since $d(y) \geq 2$, there exists $y' \in N(y) \setminus \{x\}$. The fact that $G$ does not contain cycles of length 5 implies that $y'$ is not adjacent to $S \cap N_{2}(x)$. Therefore, $S_{x} = ( S \cap N_{2}(x) ) \cup \{y'\}$ is a witness that $x$ is not shedding. Similarly,  there exists $x' \in N(x) \setminus \{y\}$, and $S_{y} = ( S \cap N_{2}(y) ) \cup \{x'\}$ is a witness that $y$ is not shedding.
\end{proof}

We show that each of the conditions of Theorem \ref{shedrec6} is necessary. If it does not hold then also the  conclusion of Theorem \ref{shedrec6} does not hold.
\begin{itemize}
\item If $N(x) \cap N(y) \neq \varnothing$ then the conclusion of Theorem \ref{shedrec6} does not necessarily hold. For example, $G$ is a union of a path  $(x_{2},x_{1},x,y,y_{1},y_{2})$ and a triangle $(x,z,y)$. In this case $xy$ is a relating edge, while $x$ and $y$ are shedding.

\item If $y$ is a leaf then the conclusion of Theorem \ref{shedrec6} does not hold. For example, $G$ is a path  $(a,b,x,y)$. In this case $xy$ is a relating edge, while $y$ is not shedding.

\item If $G$ contains a cycle of length 4 then the conclusion of Theorem \ref{shedrec6} does not necessarily hold. For example, if $G$ is a union of a path $(x_{2},x_{1},x,y,y_{1},y_{2})$ and a cycle $(x',x,y,y')$ then $xy$ not relating while $x$ and $y$ are not shedding vertices.

\item If $G$ is a copy of $C_{5}$ and $xy \in E(G)$ then the conclusion of Theorem \ref{shedrec6} does not hold. The edge $xy$ is relating, while $x$ and $y$ are shedding vertices.

\item If $G$ contains a cycle of length 6 then the conclusion of Theorem \ref{shedrec6} does not necessarily hold. For example, if $G$ is a union of cycles $(x''_{1},x'_{1},x,y,y'_{1},y''_{1})$ and  $(x''_{2},x'_{2},x,y,y'_{2},y''_{2})$ then $xy$ is not relating while $x$ and $y$ are not shedding vertices.
\end{itemize}

\section{Hereditary families of graphs}
Let $P$ be a property of graphs. Then $P$ is a {\it hereditary property} if for every graph $G$ which satisfies $P$ all induced subgraphs of $G$ satisfy $P$. Similarly, a {\it hereditary family} of graphs is a family $F$ such that if $G \in F$ then all induced subgraphs of $G$ belong to $F$. For example, claw-free graphs and bipartite graphs are hereditary families, while the set of connected graphs is not a hereditary family.

\begin{theorem}
\label{w2heredwc}
Let $F$ be a hereditary family of graphs, such that there exists an O(f(n))-time algorithm for deciding whether an input graph $G \in F$ is well-covered. Then deciding for an input graph $G \in F$ whether it belongs to $\mathbf{W_2}$ can be done in $O(nf(n))$ time.
\end{theorem}

\begin{proof}
If $G \approx P_{3}$ or $G$ contains isolated vertices then obviously $G$ is not in $\mathbf{W_2}$.

Suppose $G \not\approx P_{3}$ and $G$ does not contain isolated vertices. By Theorem \ref{w2shed}, $G$ belongs to $\mathbf{W_2}$ if and only if for every vertex $v \in V(G)$ the graph $G \setminus v$ is well-covered. Hence, an algorithm which decides whether a graph belongs to $\mathbf{W_2}$ invokes $n$ times the algorithm for recognizing well-covered graphs.
\end{proof}

\begin{theorem}
\label{w2heredshed}
Let $F$ be a hereditary family of graphs. Assume that there exists an O(f(n))-time algorithm which receives as its input a graph $G \in F$ and a vertex $v \in V(G)$, and decides whether $v \in Shed(G)$. Then deciding for an input well-covered graph $G \in F$ whether it belongs to $\mathbf{W_2}$ can be done in $O(nf(n))$ time.
\end{theorem}

\begin{proof}
If $G$ contains isolated vertices then obviously $G$ is not in $\mathbf{W_2}$. Suppose $G$ does not contain isolated vertices. By Theorem \ref{w2wcshed}, $G$ belongs to $\mathbf{W_2}$ if and only if $Shed(G) = V(G)$. Hence, an algorithm which decides whether $G$ belongs to $\mathbf{W_2}$ invokes $n$ times the algorithm for recognizing shedding vertices.
\end{proof}

\section{Claw-free graphs}
This section considers the family of claw-free graphs. It supplies polynomial algorithms for recognizing shedding vertices and $\mathbf{W_2}$ graphs.

\begin{theorem}
\label{shedclaw}
There exists an $O(n^3)$ time algorithm which receives as its input a claw-free graph, $G$, and a vertex $v \in V(G)$, and decides whether $v$ is shedding.
\end{theorem}

\begin{proof}
Define a weight function $w:N_{2}(v) \longrightarrow Z$ by $w(x) = |N(x) \cap N(v)|$, for every $x \in N_{2}(v)$. Let $S$ be an independent set of $N_{2}(v)$. We prove that distinct vertices of $S$ dominate distinct vertices of $N(v)$. Suppose on the contrary that $s_{1} \in S$ and $s_{2} \in S$ have a common neighbor $v' \in N(v)$. Then $G[\{v',v,s_{1},s_{2}\}]$ is a claw, which is a contradiction. Hence, $w(S) = \sum_{s \in S} w(s) = \sum_{s \in S} |N(s) \cap N(v)| = |N(S) \cap N(v)|$. If follows that the weight of every independent set of $N_{2}(v)$ can not exceed $|N(v)|$.

Invoke an algorithm for finding a maximum weight independent set in a claw-free graph. First such an algorithm is due to Minty \cite{minty:claw}, while the best known one with the complexity $O(n^3)$ may be found in \cite{fea:claw}. The algorithm outputs a set $S^{*} \subseteq N_{2}(v)$. If $w(S^{*}) = |N(v)|$ then $S^{*}$ is an independent set of $N_{2}(v)$ which dominates $N(v)$, and therefore $v$ is not shedding. However, if $w(S^{*}) < |N(v)|$ then there does not exist an independent set of $N_{2}(v)$ which dominates $N(v)$, and $v$ is shedding. Deciding whether $v$ is shedding can be done in $O(n^3)$ time.
\end{proof}

\begin{corollary}
\label{w2wcclaw}
There exists an $O(n^4)$ time algorithm which receives a well-covered claw-free graph, and decides whether it belongs to $\mathbf{W_2}$.
\end{corollary}

\begin{proof}
Follows immediately from Theorem \ref{shedclaw} and Theorem \ref{w2heredshed}.
\end{proof}

Deciding whether an input claw-free graph is well-covered can be done in $O(m^{\frac{3}{2}}n^{3})$ time \cite{lt:equimatchable}. Therefore, Theorem \ref{w2heredwc} implies that recognizing $\mathbf{W_2}$ claw-free graphs can be completed in $O(m^{\frac{3}{2}}n^{4})$ time. However, Theorem \ref{w2claw} supplies a more efficient algorithm which solves this problem.

\begin{theorem}
\label{w2claw}
There exists an $O(m^{\frac{3}{2}}n^{3})$ time algorithm for recognizing $\mathbf{W_2}$ claw-free graphs.
\end{theorem}

\begin{proof}
Invoke the algorithm of \cite{lt:equimatchable} for deciding whether the input graph $G$ is well-covered. If $G$ is not well-covered then $G \not\in \mathbf{W_2}$. If $G$ is well-covered, invoke the algorithm of Corollary \ref{w2wcclaw} to decide whether $G$ belongs to $\mathbf{W_2}$.

The complexity of recognizing $\mathbf{W_2}$ claw-free graphs is $O(m^{\frac{3}{2}}n^{3} + n^4) = O(m^{\frac{3}{2}}n^{3})$.
\end{proof}

\section{Graphs without cycles of length 5}
\subsection{Shedding vertices}
\begin{theorem}
\cite{lt:wwc456}
\label{dc5} 
Let $G$ be a graph without cycles of length 5, and $v \in V(G)$. Then the following holds.
\begin{enumerate}
\item Every maximal independent set of $N_{2}(v)$ dominates $N(v) \cap N(N_{2}(v))$.
\item If $N_{2}(v)$ dominates $N(v)$ then every maximal independent set of $N_{2}(v)$ dominates $N(v)$. 
\end{enumerate}
\end{theorem}

\begin{corollary}
\label{c5shed2}
Let $G$ be a graph without cycles of length 5. Let $v \in V(G)$. Then $v$ is shedding if and only if  $N_{2}(v)$ does not dominate $N(v)$.
\end{corollary}

\begin{corollary}
\label{c5shedalg} The following problem can be solved in $O(n+m)$ time. \newline
Input: A graph $G$ without cycles of length 5, and a vertex $v \in V(G)$. \newline
Question: Is $v$ shedding ?
\end{corollary}

\begin{proof}
By Corollary \ref{c5shed2}, $v$ is shedding if and only if $N_{2}(v)$ does not dominate $N(v)$. Since one can decide in $O(n+m)$ time whether $N_{2}(v)$ dominates $N(v)$, it is possible to decide in the same complexity whether $v$ is a shedding vertex.
\end{proof}

\begin{corollary}
\label{w2wcc5}
The following problem can be solved in $O(n(n+m))$ time. \newline
Input: A well-covered graph $G$ without cycles of length 5. \newline
Question: Is $G$ in the class $\mathbf{W_2}$ ?
\end{corollary}

\begin{proof}
Follows immediately from Theorem \ref{w2heredshed} and Corollary \ref{c5shedalg}.
\end{proof}

\subsection{Graphs without cycles of lengths 3 and 5}
\begin{theorem}
\label{w2vc3c5}
Let $G \in \mathbf{W_2}$ such that $G \not\approx K_{2}$. Let $v \in V(G)$. Then at least one of the following holds.
\begin{enumerate}
\item $v$ is on a triangle $(v,v_{1},v_{2})$ such that $\{v_{1},v_{2}\}$ is not dominated by $N_{2}(v)$.
\item $v$ is on a $C_{5}$.
\end{enumerate}
\end{theorem}

\begin{proof}
The fact that $G \in \mathbf{W_2}$ and $G \not\approx K_{2}$ implies that $G$ does not contain leaves. Two cases are considered.

{\sc case 1: $N_{2}(v)$ does not dominate $N(v)$} There exists $v_{1} \in N(v)$ which is not dominated by $N_{2}(v)$. Since $v_{1}$ is not a leaf, there exists $v_{2} \in N(v) \cap N(v_{1})$. Condition 1 holds.

{\sc case 2: $N_{2}(v)$ dominates $N(v)$.} Let $N(v) = \{v_{1},\ldots,v_{k}\}$. For every $1 \leq i \leq k$ let $v'_{i} \in N_{2}(v) \cap N(v_{i})$. Since $v$ is a shedding vertex, $\{v'_{1},\ldots,v'_{k}\}$ is not independent. Hence, $(v'_{i},v'_{j}) \in E(G)$ for some $1 \leq i < j \leq k$. Therefore, $(v,v_{i},v'_{i},v'_{j},v_{j})$ is a $C_{5}$ which contains $v$. Condition 2 holds.
\end{proof}

\begin{corollary}
\label{w2bipartite}
Let $G$ be a graph without cycles of lengths 3 and 5. Then $G \in \mathbf{W_2}$ if and only if $G \approx K_{2}$.
\end{corollary}

\subsection{Graphs without cycles of lengths 4 and 5}
In \cite{fhn:wc45}, the following definitions are used. A graph $G$ belongs to the family $F$ if there exists $\{x_{1},x_{2},\ldots,x_{k}\} \subseteq V(G)$, where for each $1 \leq i \leq k$ it holds that $x_{i}$ is simplicial, $|N[x_{i}]| \leq 3$, and $\{ N[x_{i}] : i = 1, 2, \ldots, k\}$ is a partition of $V(G)$.

\begin{theorem}
\label{wc45}
\cite{fhn:wc45}
Let $G$ be a connected well-covered graph containing neither $C_{4}$ nor $C_{5}$ as a subgraph. Then
\begin{enumerate}
\item $G$ contains a shedding vertex or $G \approx K_{1}$ if and only if $G \in F$.
\item Otherwise, $G$ is isomorphic to $C_{7}$ or $T_{10}$. (See Figure \ref{T10}.)
\end{enumerate}
\end{theorem}

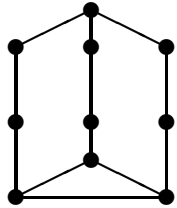
\begin{figure}[h]
\setlength{\unitlength}{1.0cm} \begin{picture}(3,3)\thicklines
\multiput(6,0.5)(0,1){3}{\circle*{0.2}}
\multiput(8,0.5)(0,1){3}{\circle*{0.2}}
\multiput(7,1.5)(0,1){2}{\circle*{0.2}}
\put(7,1){\circle*{0.2}}
\put(7,3){\circle*{0.2}}
\put(6,0.5){\line(0,1){2}}
\put(6,0.5){\line(1,0){2}}
\put(6,0.5){\line(2,1){1}}
\put(8,0.5){\line(0,1){2}}
\put(7,1){\line(0,1){2}}
\put(7,1){\line(2,-1){1}}
\put(6,2.5){\line(2,1){1}}
\put(7,3){\line(2,-1){1}}
\end{picture}\caption{The graph $T_{10}$.}%
\label{T10}
\end{figure}

Theorem \ref{w2c45} is the main result of this section.

\begin{theorem}
\label{w2c45}
Let $G$ be a connected graph containing neither $C_{4}$ nor $C_{5}$ as a subgraph. The following conditions are equivalent.
\begin{enumerate}
\item $G \in \mathbf{W_2}$
\item $G \approx K_{2}$ or there exists a partition $\{T_{1}, T_{2}, \ldots, T_{k}\}$ of $V(G)$, such that $G[T_{i}]$ is a triangle with at least 2 simplicial vertices, for each $1 \leq i \leq k$.
\end{enumerate}
\end{theorem}

\begin{proof}

$1 \implies 2$ Suppose that $G \in \mathbf{W_2}$. By Theorem \ref{w2wcshed}, $V(G) = Shed(G)$. Therefore, $G$ is isomorphic neither to $K_{1}$ nor to $C_{7}$ nor to $T_{10}$. By Theorem \ref{wc45}, $G \in F$. Hence, there exists $\{x_{1},x_{2},\ldots,x_{k}\} \subseteq V(G)$, where for each $1 \leq i \leq k$ it holds that $x_{i}$ is simplicial, $|N[x_{i}]| \leq 3$, and $\{ N[x_{i}] : i = 1, 2, \ldots, k\}$ is a partition of $V(G)$.

If $|N[x_{i}]| = 1$ then $x_{i}$ is an isolated vertex, which is a contradiction. If $|N[x_{i}]| = 2$ then $x_{i}$ is a leaf, and therefore $G \approx K_{2}$. In the remaining case $|N[x_{i}]| = 3$ for every $1 \leq i \leq k$. Let $y_{i}$ and $z_{i}$ be the neighbors of $x_{i}$, and assume on the contrary that both are not simplicial. There exists $y'_{i} \in N(y_{i}) \setminus \{x_{i}, z_{i}\}$ and $z'_{i} \in N(z_{i}) \setminus \{x_{i}, y_{i}\}$. if $y'_{i} = z'_{i}$ then $(y'_{i}, y_{i}, x_{i}, z_{i})$ is a $C_{4}$, which is a contradiction. If $y'_{i} \in N(z'_{i})$ then $(z'_{i}, y'_{i}, y_{i}, x_{i}, z_{i})$ is a $C_{5}$, which is a contradiction, again. Therefore, $y'_{i}$ and $z'_{i}$ are distinct and non-adjacent.

Define $A = \{x_{i}\}$ and $B = \{y'_{i},z'_{i}\}$. Clearly, $A$ and $B$ are independent and disjoint. Therefore, they are contained in two disjoint maximal independent sets, $A^{*}$ and $B^{*}$, respectively. However, $x_{i} \not \in B^{*}$ since $x_{i} \in A^{*}$. Moreover, $y_{i}$ and $z_{i}$, are not in $B^{*}$, because they are dominated by $\{y'_{i}, z'_{i}\} \subseteq B^{*}$. Therefore $x_{i} \not\in N[B^{*}]$, which is a contradiction. At least one of $y_{i}$ and $z_{i}$ is simplicial.

$2 \implies 1$ Suppose that Condition 2 holds and $G \not\approx K_{2}$. Then $|S \cap T_{i}| \leq 1$ for every independent set $S \subseteq V(G)$ and for every $1 \leq i \leq k$. Moreover, every dominating independent set contains exactly $k$ vertices. If $A$ and $B$ are disjoint independent sets of $G$, then there exist $A \subseteq A^{*}$ and $B \subseteq B^{*}$ such that $A^{*}$ and $B^{*}$ are disjoint maximal independent sets of $G$. Therefore, $G \in \mathbf{W_2}$.
\end{proof}

\section{Graphs with girth at least 5}
In \cite{fh:wcg5} well-covered graphs with girth at least 5 are characterized. The following notation and definitions are used. A {\it leaf} is a vertex of degree 1. A {\it steam} is a vertex adjacent to a leaf. An edge connecting a leaf and a steam is called {\it pendant}. A 5-cycle $C$ is called {\it basic} if $C$ does not contain two adjacent vertices of degree 3 or more. A graph $G$ belongs to the family $\mathbf{PC}$ if $V(G)$ can be partitioned into two subsets, $P$ and $C$, such that the following holds.
\begin{itemize}
\item $P$ contains the vertices adjacent to the pendant edges.
\item The pendant edges are a perfect matching of $P$.
\item $C$ contains the vertices of the basic 5-cycles.
\item The basic 5-cycles are a partition of $C$.
\end{itemize}

\begin{theorem}
\label{wcg5}
\cite{fh:wcg5}
Let $G$ be a graph with girth 5 or more. Then the following conditions are equivalent.
\begin{enumerate}
\item $G$ is well-covered and contains a shedding vertex.
\item $G$ belongs to the family $\mathbf{PC}$.
\end{enumerate}
\end{theorem}

Theorem \ref{w2g5} is the main result of this section.

\begin{theorem}
\label{w2g5} The only graphs with girth 5 or more in the class $\mathbf{W_2}$ are $C_{5}$ and $K_{2}$.
\end{theorem}

\begin{proof}
Let $G$ be a graph with girth 5 or more in the class $\mathbf{W_2}$. Obviously, $G$ is well-covered. By Theorem \ref{w2wcshed}, $Shed(G) = V(G)$. By Theorem \ref{wcg5} $G$ belongs to the family $\mathbf{PC}$. The only graph in $\mathbf{W_2}$ which contains leaves is $K_{2}$. Suppose that  $G$  does not contain leaves and belongs to $\mathbf{W_2}$. Then $P=\varnothing$.

Let $C=(v_{1},\ldots,v_{5})$ be a basic cycle in $G$, and assume on the contrary that $d(v_{1}) \geq 3$. Then $d(v_{2})=d(v_{5})=2$. Let $v_{1}' \in N(v_{1}) \setminus \{v_{2},v_{5}\}$. Define $A=\{v_{1}'\}$ and $B=\{v_{2},v_{5}\}$. Since $A$ and $B$ are disjoint and independent, they are contained in two disjoint maximal independent sets, $A^{*}$ and $B^{*}$, of $G$. Since $v_{2}$ and $v_{5}$ belong to $B^{*}$, they do not belong to $A^{*}$. Hence, $v_{2}$ and $v_{5}$ are neighbors of vertices belonging to $A^{*}$. However, $v_{1} \not\in A^{*}$. Hence, $\{v_{3}, v_{4}\} \subseteq A^{*}$. Consequently, $A^{*}$ is not independent, which is a contradiction.

Therefore, the degree of every vertex of $C$ is 2, and $G \approx C_{5}$.
\end{proof}

\begin{theorem}
\label{shedwcg5} 
There exists an $O(n)$ algorithm which solves the following problem.\\
Input: A connected well-covered graph $G$ with girth 5 or more and a vertex $v \in V(G)$.\\
Question: Is $v$ is shedding.
\end{theorem}

\begin{proof}
Let $I=(G,v)$ be an instance of the problem. By Theorem \ref{wcg5}, if $G$ does not belong to the family $\mathbf{PC}$ then $v$ not is shedding. Suppose $G$ belongs to the family $\mathbf{PC}$. Four cases are possible.

{\sc Case 1: $v$ is a leaf.} Obviously, $v$ is not shedding.

{\sc Case 2: $v$ is a neighbor of a leaf.} Obvously, $v$ is shedding.

{\sc Case 3: $v$ is on a basic cycle, $C$, and $d(v)=2$.} Let $N(v) = \{v_{1}, v_{2}\}$. The vertex $v$ is shedding if and only if $N(\{v_{1},v_{2}\}) \setminus V(C) \neq \varnothing$. 

{\sc Case 4: $v$ is on a basic cycle, $C$, and $d(v) \geq 3$.} There exist two vertices, $v_{1}$ and $v_{2}$, in $C$ which are adjacent to $v$. It holds that $d(v_{1}) = d(v_{2}) = 2$. There does not exist an independent set of $N_{2}(v)$ which dominates $\{v_{1},v_{2}\}$. Therefore, $v$ is shedding.

In \cite{fh:wcg5} a finite list of all connected well-covered graphs with girth at least 5 and without shedding vertices is presented. Hence, deciding whether $G$ belongs to the family $\mathbf{PC}$ can be completed in $O(1)$ time. Deciding which case holds for $v$ can be done in $O(n)$ time. If Case 3 holds then $O(n)$ time is needed to decide whether the condition $N(\{v_{1},v_{2}\}) \setminus V(C) \neq \varnothing$ holds. Otherwise, deciding whether $v$ is shedding can be done in $O(1)$ time. The total complexity of the algorithm is $O(n)$.
\end{proof}

\section{Well-covered graphs without cycles of lengths 4 and 6}
\begin{theorem}
\label{shedc46} The following problem can be solved in $O(n(n+m))$ time. \newline
Input: A graph $G$ without cycles of lengths 4 and 6, and a vertex $v \in V(G)$. \newline
Question: Is $v$ shedding ?
\end{theorem}

\begin{proof} Every vertex of $N_{2}(v)$ is adjacent to
exactly one vertex of $N(v)$, or otherwise $G$ contains a $C_{4}$.
Every component of $N_{2}(v)$ contains at most $2$ vertices, or
otherwise $G$ contains either a $C_{4}$ or a $C_{6}$. Let $A_{1},\ldots,A_{k}$ be the components of $N_{2}(v)$.

The following algorithm decides whether $v$ is shedding. It constructs a flow network $F_{v} = \{ G_{F} = (V_{F}, E_{F}), s \in V_{F}, t \in V_{F}, c:E_{F}\longrightarrow \mathbb{R} \}$. A similar construction was used in \cite{lt:relatedc4}.

Define $V_{F} = N(v) \cup N_{2}(v) \cup \{a_{1},\ldots ,a_{k}, s, t\}$, where
$a_{1},\ldots,a_{k}, s, t$ are new vertices, $s$ and $t$ are the source and sink of
the network, respectively. Denote $A = \{a_{1}, \ldots , a_{k}\}$.

The directed edges $E_{F}$ are:

\begin{itemize}
\item the directed edges from $s$ to each vertex of $N(v)$.

\item all directed edges $v_{1}v_{2}$ s.t. $v_{1}\in N(v)$, $v_{2} \in N_{2}(v)$ and $v_{1}v_{2}\in E$.

\item the directed edges $za_{i}$, for each $1 \leq i \leq k$ and for each $z \in A_{i}$.

\item the directed edges $a_{i}t$, for each $1\leq i\leq k$.
\end{itemize}

Let $c \equiv 1$. Invoke Ford and Fulkerson's algorithm for finding a maximum flow $f:E_{F}\longrightarrow\mathbb{R}$ in the network. The flow in a vertex $x \in V_{F}$ is
defined by: $\Sigma_{(u,x) \in E_{F}} f(u,x)$. Let $S_{v}$ be the set of
vertices in $N_{2}(v)$ in which there is a positive flow. One can prove that $S_{v}$ is a maximum independent set of $N_{2}(v)$, Ford and Fulkerson's algorithm terminates after $|S_{v}|$ iterations, and $|f| = |S_{v}| = |N(v) \cap N(S_{v})|$. For more details see \cite{lt:relatedc4}.

If $|S_{v}| = |N(v)|$ then $S_{v}$ dominates $N(v)$, and therefore $v$ is not shedding. However, if $|S_{v}| < |N(v)|$ then $S_{v}$ does not dominate $N(v)$, and there
does not exist an independent set of $N_{2}(v)$ which dominates $N(v)$. In this case the algorithm terminates announcing that $v$ is shedding.

Every iteration of Ford and Fulkerson's algorithm takes $O(n+m)$ time. In the restricted case of this context, the algorithm of Ford and Fulkerson terminates after $O(n)$ iterations. Hence, total complexity of this algorithm is $O(n(n+m))$. 
\end{proof}

\begin{corollary}
\label{w2wcc46} The following problem can be solved in $O(n^{2}(n+m))$ time. \newline
Input: A well-covered graph $G$ without cycles of lengths 4 and 6. \newline
Question: Is $G$ in the class $\mathbf{W_2}$ ?
\end{corollary}

\begin{proof}
Follows immediately from Theorems \ref{w2heredshed} and \ref{shedc46}. 
\end{proof}

Note that the complexity status of recognizing well-covered graphs without cycles of lengths 4 and 6 is not known. If recognizing well-covered graphs without cycles of lengths 4 and 6 can be done polynomially then recognizing  graphs in $\mathbf{W_2}$ without cycles of lengths 4 and 6 is a polynomial task, too.

\section{$\alpha(G) = k$}
This section considers graphs with $\alpha(G) = k$. It supplies polynomial algorithms for recognizing well-covered graphs, shedding vertices and $\mathbf{W_2}$ graphs.

\begin{theorem}
\label{wcalpha} Let $k \geq 2$. The following problem can be solved in $O(n^{k-1})$ time. \newline
Input: A graph $G$ with $\alpha(G) = k$. \newline
Question: Is $G$ well-covered?
\end{theorem}

\begin{proof}
The following algorithm solves the problem. For every set $S \subseteq V(G)$ such that $|S| < k$, decide whether $S$ is independent and dominates $V(G)$. Once a dominating independent set is found, the algorithm terminates announcing that $G$ is not well-covered. However, if all sets of size smaller than $k$ were checked, and none of them is dominating and independent then $G$ is well-covered.

The number of sets which are checked is $O(n^{k-1})$. For each such a set, deciding whether it is independent and dominating takes $O(k^{2})$ time, which is constant for every fixed $k$. Therefore, deciding whether $G$ is well-covered can be done in $O(n^{k-1})$ time.
\end{proof}

\begin{theorem}
\label{shedalpha} Let $k \geq 2$. The following problem can be solved in $O(n^{k})$ time. \newline
Input: A graph $G$ with $\alpha(G) = k$ and a vertex $v \in V(G)$. \newline
Question: Is $v$ a shedding vertex?
\end{theorem}

\begin{proof}
Since $\alpha(G) = k$, an independent set of $N_{2}(v)$ contains at most $k-1$ vertices.

The following algorithm solves the problem. For every set $S \subseteq N_{2}(v)$ such that $|S| \leq k-1$, decide whether $S$ is independent and dominates $N(v)$. Once an independent set of $N_{2}(v)$ which dominates $N(v)$ was found, the algorithm terminates announcing that $v$ is not shedding. If all possible sets were checked, and none of them is independent and dominates $N(v)$ then $v$ is shedding.

The number of sets which are checked is $O(n^{k-1})$. For each such a set, deciding whether it is independent and dominates $N(v)$ takes $O(n)$ time. Therefore, deciding whether $v$ is shedding can be done in $O(n^{k})$ time.
\end{proof}

It follows from Theorem \ref{w2heredwc} and Theorem \ref{wcalpha} that recognizing $\mathbf{W_2}$ graphs with $\alpha(G) = k$ can be done in $O(n^{k})$ time. Theorem \ref{w2complexity} finds another algorithm which solves this problem with the same complexity.

\begin{theorem}
\label{w2complexity} Let $k \geq 2$. The following problem can be solved in $O(n^{k})$ time. \newline
Input: A graph $G$ with $\alpha(G) = k$. \newline
Question: Is $G \in \mathbf{W_2}$ ?
\end{theorem}

\begin{proof}
If $G$ contains isolated vertices then obviously $G \not\in \mathbf{W_2}$. Otherwise, check all sets of vertices with size smaller than $k$. For each such a set, $S$, if $S$ is independent and dominates $V(G)$, then $G$ is not well-covered, and therefore it is not in $\mathbf{W_2}$. If $S$ is independent and $N[S] = V(G) \setminus \{v\}$ for some $v \in V(G)$, then $v$ is not shedding, and therefore $G \not\in \mathbf{W_2}$. However, if $|V(G) \setminus N[S]| \geq 2$ for every independent set $S$ with size smaller than $k$, then $G$ is well-covered and all vertices are shedding. By Theorem \ref{w2wcshed}, $G$ belongs to $\mathbf{W_2}$.

This algorithm considers $O(n^{k-1})$ sets. Each set is checked in $O(n^{k})$ time. Hence, the algorithm terminates in $O(n^{k})$ time.
\end{proof}

\section{Conclusions}
Let $\mathcal{G}(\widehat{C_{i_{1}}},\ldots,\widehat{C_{i_{k}}})$ be the family of
all graphs which do not contain $C_{i_{1}}$,\ldots,$C_{i_{k}}$. The following table presents complexity results concerning the five major problems presented in this paper. The empty table cells correspond to unsolved cases.

\begin{center}%
\begin{tabular}
[c]{|c|c|c|c|c|c|}\hline
\textbf{{Input}} & $\mathbf{WC}$ & $\mathbf{{W2}}$ & $\mathbf{WCW2}$ & $\mathbf{SHED}$ & $\mathbf{WCSHED}$ \\\hline%
\begin{tabular}
[c]{c}%
\\
general
\end{tabular}
&
\begin{tabular}
[c]{c}%
\textbf{co}-\textbf{NPC}\\
\cite{cs:note,sknryn:compwc}%
\end{tabular}
&
\begin{tabular}
[c]{c}%
\\
\end{tabular}
&
\begin{tabular}
[c]{c}%
\\
\end{tabular}
&
\begin{tabular}
[c]{c}%
\textbf{co}-\textbf{NPC}\\
Th. \ref{shedc3conpc}
\end{tabular}
&
\begin{tabular}
[c]{c}
\end{tabular}
\\\hline
\begin{tabular}
[c]{c}%
\\
claw-free
\end{tabular}
&
\begin{tabular}
[c]{c}%
\textbf{P}\\
\cite{tata:wck13f}
\end{tabular}
&
\begin{tabular}
[c]{c}%
\textbf{P}\\
Th. \ref{w2claw}
\end{tabular}
&
\begin{tabular}
[c]{c}%
\textbf{P}\\
Cor. \ref{w2wcclaw}
\end{tabular}
&
\begin{tabular}
[c]{c}%
\textbf{P}\\
Th. \ref{shedclaw}
\end{tabular}
&
\begin{tabular}
[c]{c}%
\textbf{P}\\
Th. \ref{shedclaw}
\end{tabular}
\\\hline
\begin{tabular}
[c]{c}%
\\
$\mathcal{G}(\widehat{C_{3}})$
\end{tabular}
&
\begin{tabular}
[c]{c}%
\end{tabular}
&
\begin{tabular}
[c]{c}%
\\
\end{tabular}
&
\begin{tabular}
[c]{c}%
\\
\end{tabular}
&
\begin{tabular}
[c]{c}%
\textbf{co}-\textbf{NPC}\\
Th. \ref{shedc3conpc}
\end{tabular}
&
\begin{tabular}
[c]{c}
\end{tabular}
\\\hline
\begin{tabular}
[c]{c}%
\\
$\mathcal{G}(\widehat{C_{3}},\widehat{C_{4}})$
\end{tabular}
&
\begin{tabular}
[c]{c}%
\textbf{P}\\
\cite{fh:wcg5}
\end{tabular}
&
\begin{tabular}
[c]{c}%
\textbf{P}\\
Th. \ref{w2g5} 
\end{tabular}
&
\begin{tabular}
[c]{c}%
\textbf{P}\\
Th. \ref{w2g5} 
\end{tabular}
&
\begin{tabular}
[c]{c}%
\\
\end{tabular}
&
\begin{tabular}
[c]{c}%
\textbf{P}\\
Th. \ref{shedwcg5} 
\end{tabular}
\\\hline
\begin{tabular}
[c]{c}%
\\
$\mathcal{G}(\widehat{C_{5}})$
\end{tabular}
&
\begin{tabular}
[c]{c}%
\\
\end{tabular}
&
\begin{tabular}
[c]{c}%
\end{tabular}
&
\begin{tabular}
[c]{c}%
\end{tabular}
&
\begin{tabular}
[c]{c}%
\textbf{P}\\
Cor. \ref{c5shedalg} 
\end{tabular}
&
\begin{tabular}
[c]{c}%
\textbf{P}\\
Cor. \ref{c5shedalg} 
\end{tabular}
\\\hline
\begin{tabular}
[c]{c}%
\\
$\mathcal{G}(\widehat{C_{3}},\widehat{C_{5}})$
\end{tabular}
&
\begin{tabular}
[c]{c}%
\\
\end{tabular}
&
\begin{tabular}
[c]{c}%
\textbf{P}\\
Cor. \ref{w2bipartite}
\end{tabular}
&
\begin{tabular}
[c]{c}%
\textbf{P}\\
Cor. \ref{w2bipartite}
\end{tabular}
&
\begin{tabular}
[c]{c}%
\textbf{P}\\
Cor. \ref{c5shedalg} 
\end{tabular}
&
\begin{tabular}
[c]{c}%
\textbf{P}\\
Cor. \ref{c5shedalg} 
\end{tabular}
\\\hline
\begin{tabular}
[c]{c}%
\\
$\mathcal{G}(\widehat{C_{4}},\widehat{C_{5}})$
\end{tabular}
&
\begin{tabular}
[c]{c}%
\textbf{P}\\
\cite{fhn:wc45}
\end{tabular}
&
\begin{tabular}
[c]{c}%
\textbf{P}\\
Th. \ref{w2c45}
\end{tabular}
&
\begin{tabular}
[c]{c}%
\textbf{P}\\
Th. \ref{w2c45}
\end{tabular}
&
\begin{tabular}
[c]{c}%
\textbf{P}\\
Cor. \ref{c5shedalg} 
\end{tabular}
&
\begin{tabular}
[c]{c}%
\textbf{P}\\
Cor. \ref{c5shedalg} 
\end{tabular}
\\\hline
\begin{tabular}
[c]{c}%
\\
$\mathcal{G}(\widehat{C_{4}},\widehat{C_{6}})$
\end{tabular}
&
\begin{tabular}
[c]{c}%
\\
\end{tabular}
&
\begin{tabular}
[c]{c}%
\\
\end{tabular}
&
\begin{tabular}
[c]{c}%
\textbf{P}\\
Cor. \ref{w2wcc46} 
\end{tabular}
&
\begin{tabular}
[c]{c}%
\textbf{P}\\
Th. \ref{shedc46}
\end{tabular}
&
\begin{tabular}
[c]{c}%
\textbf{P}\\
Th. \ref{shedc46}
\end{tabular}
\\\hline
\end{tabular}

\end{center}

\section{A conjecture}
A vertex $v \in V(G)$ is shedding if and only if for every independent set $S \subseteq N_{2}(v)$ it holds that $N(v) \setminus N(S) \neq \varnothing$. A restricted case of a shedding vertex is a vertex $v$ such that for every maximal independent set $S$ of $V(G) \setminus N[v]$ it holds that $\alpha(N(v) \setminus N(S)) = 1$. Assume that $G$ is in the class $\mathbf{W_2}$, $v \in V(G)$, and $S$ is a maximal independent set of $V(G) \setminus N[v]$. Then $H = G[V(G) \setminus N[S]] = G[N[v] \setminus N(S)]$ is well-covered. However, $\{v\}$ is a maximal independent set of $H$. Hence, $H$ is a clique. Therefore, Theorem  \ref{w2wcshed} is an instance of Theorem \ref{w2alpha}.

\begin{theorem}
\label{w2alpha}
For every well-covered graph $G$ the following assertions are equivalent:
\begin{enumerate}
\item $G$ is in the class $\mathbf{W_2}$.
\item $\alpha(N(v) \setminus N(S)) = 1$ for every vertex $v \in V(G)$ and for every maximal independent set $S$ of $V(G) \setminus N[v]$.
\end{enumerate}
\end{theorem}

We conjecture that Theorem \ref{w2alpha} holds for general graphs.

\begin{conjecture}
\label{w2conj}
For every graph $G$ the following assertions are equivalent:
\begin{enumerate}
\item $G$ is in the class $\mathbf{W_2}$.
\item $\alpha(N(v) \setminus N(S)) = 1$ for every vertex $v \in V(G)$ and for every maximal independent set $S$ of $V(G) \setminus N[v]$.
\end{enumerate}
\end{conjecture}

\end{document}